\documentclass[letterpaper,11pt]{article}
\usepackage{amsmath, amsthm, amssymb}    
\DeclareMathOperator{\arccot}{arccot}
\newtheorem{theorem}{Theorem}
\newtheorem{proposition}[theorem]{Propsition}
\newtheorem{lemma}[theorem]{Lemma}
\usepackage{bridges}
\usepackage{graphicx}
\usepackage[colorlinks=true, urlcolor=blue, citecolor=black, linkcolor=black]{hyperref}  
\usepackage[labelformat = simple]{subcaption}

\title{Fibonacci and Lucas Sequences in Aperiodic Monotile Supertiles}

\author{Shiying Dong
\vspace{10pt}\\
Greenwich, Connecticut, USA; shiyingdong@gmail.com
}

\date{\today}				

\begin{document}

\maketitle

\thispagestyle{empty}

\begin{abstract}
This paper first discusses the size and orientation of hat supertiles. Fibonacci and Lucas sequences, as well as a third integer sequence linearly related to the Lucas sequence are involved. The result is then generalized to any aperiodic tile in the hat family. 

\end{abstract}

\section*{Introduction}
2023 saw one of the most exciting discoveries in mathematics: the aperiodic monotiles~\cite{hat}\cite{spectre}, commonly known as \emph{hat} and \emph{spectre} tiles. One of their fascinating aspects is the tiling process, which can be generated using recursively defined supertiles. Although complex, each supertile carries a vector indicating its size and natural direction. In this paper, we calculate these vectors for the supertiles in the hat family, yielding a simple expression involving Fibonacci and Lucas sequences. Moreover, we discover an integer sequence linearly related to the Lucas sequence in the rotation of the supertiles.

\section*{Supertiles}
In this section, we will review the supertile forming rules. We start from the single hat and the two-hat compound (Figure~\ref{fig:1}), with a vector called $V_1$ marked in each piece. The two pieces are arranged in Figure~\ref{fig:1} so that their $V_1$'s are parallel. 
We define the head and tail of the hat piece as the head and tail of its $V_1$—the same for the two-hat compound. Let's call these tiles \emph{hat-1} and \emph{thc-1}, the \emph{first generation of supertiles}.
\begin{figure}[h!tbp]
\centering
\begin{minipage}[b]{0.35\textwidth} 
\centering
	\includegraphics[scale = 1]{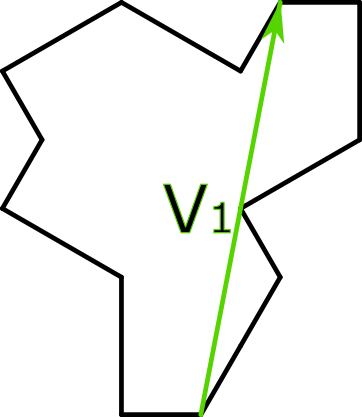}
        	\subcaption{Hat}
        	\label{fig:1a}
\end{minipage}
\hspace{3em}
\begin{minipage}[b]{0.35\textwidth} 
\centering
	\includegraphics[scale = 1]{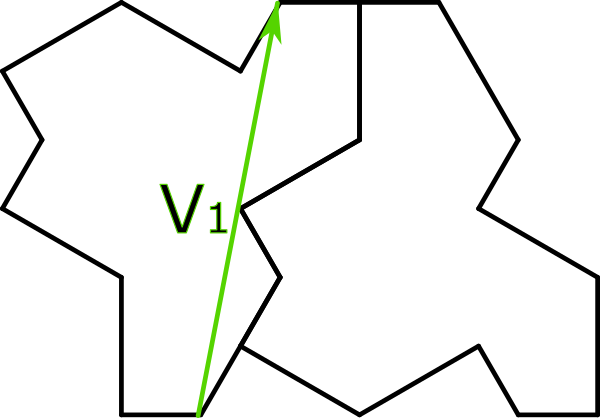}
        	\subcaption{Two-hat compound} 
	\label{fig:1b}
\end{minipage}
\caption{Marked hat tiles.}
\label{fig:1}
\end{figure}

The next generation of supertiles is formed by the following pattern. There is a two-hat compound piece and a first hat piece with its $V_1$ $120$-degree clockwise from that of the two-hat compound, and the two pieces share the same tail. The next two hat pieces have their tails at the previous hat pieces' heads and rotated $60$ degrees counterclockwise from the previous piece. The fourth is parallel to the third and is snapped in between the two-hat compound and the third hat piece. The next two hat pieces have their tails at the previous hat pieces' heads and rotated $60$ degrees counterclockwise from the previous piece. This supertile has its \emph{supervector} called $V_2$, shown in Figure~\ref{fig:2}. This forms the second generation supertile \emph{hat-2} of the hat tile. The supertile of the two-hat compound \emph{thc-2} is the same except for the missing third hat piece.
\begin{figure}[h!tbp]
\centering
\begin{minipage}[b]{0.35\textwidth} 
\centering
	\includegraphics[scale = 0.5]{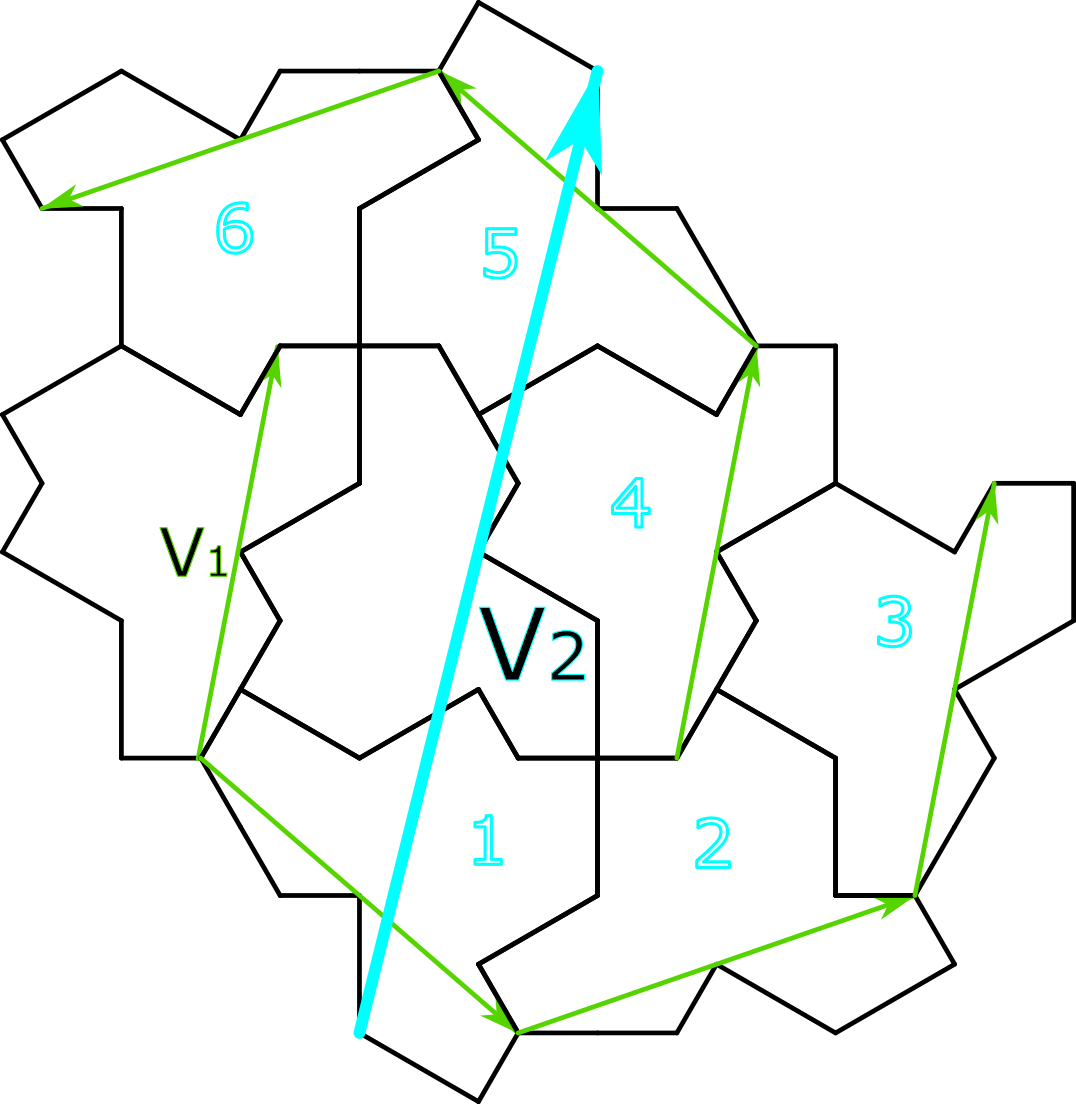}
        	\subcaption{Hat supertile \emph{hat-2}}
        	\label{fig:2a}
\end{minipage}
\hspace{3em}
\begin{minipage}[b]{0.35\textwidth} 
\centering
	\includegraphics[scale = 0.5]{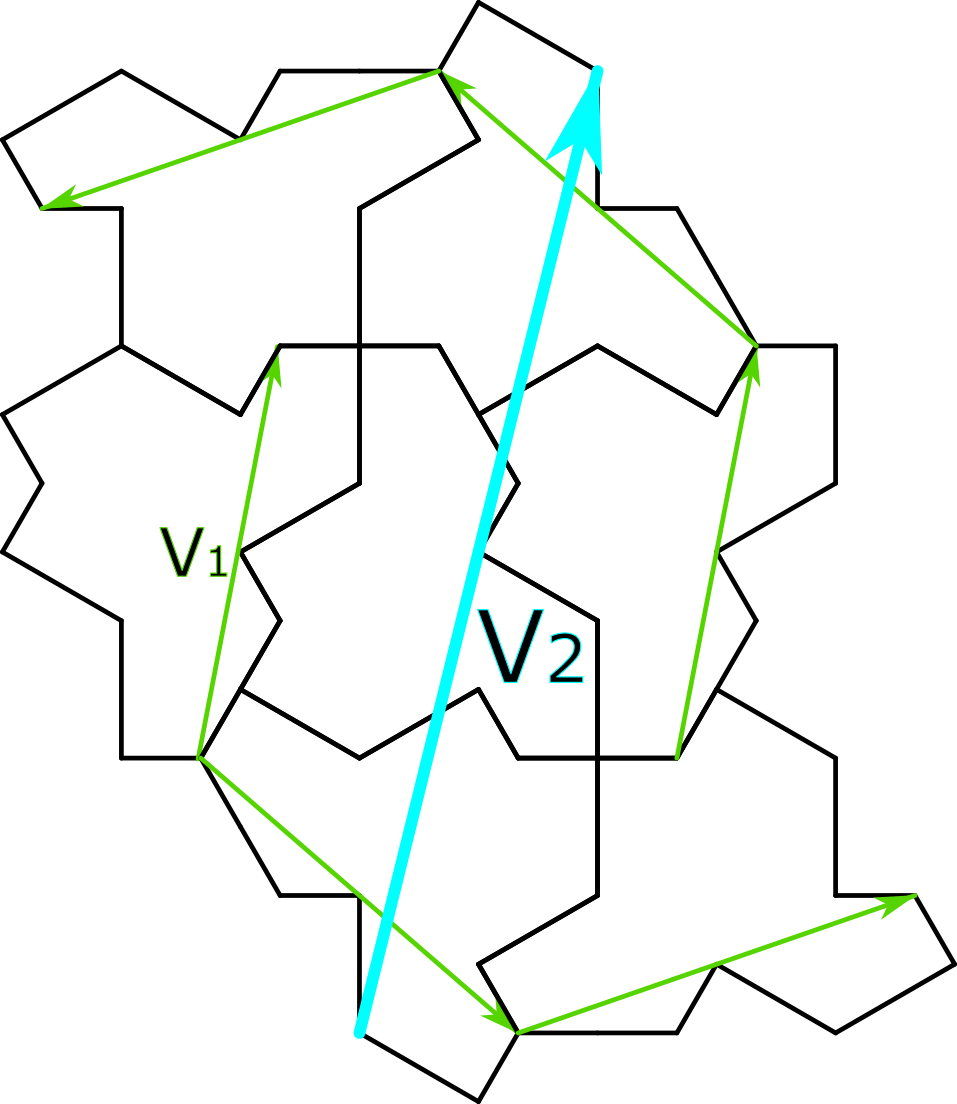}
        	\subcaption{Two-hat compound supertile \emph{thc-2}} 
	\label{fig:2b}
\end{minipage}
\caption{The second generation supertiles.}
\label{fig:2}
\end{figure}

The previous paragraph also describes how the $n$th ($n\geq 3$) generation supertiles are formed using the prior supertiles. To see how the supervectors $V_{n+1}$ are formed, we need to look at the case of $n=3$. 
From Figure~\ref{fig:3}, we can see that thc-2 and the fourth hat-2 meet in the way that the thc-1 of the latter has its supervector exactly where that of the missing third hat-1 of thc-2 would be. The second observation is that $V_3$ starts from the head of the third hat-1 of the first hat-2 and ends at the head of the third hat-1 of the fifth hat-2. In general, these two facts work for all generations above this one. 
\begin{figure}[h!tbp]
\centering
\begin{minipage}[b]{0.4\textwidth} 
\centering
	\includegraphics[scale = 0.25]{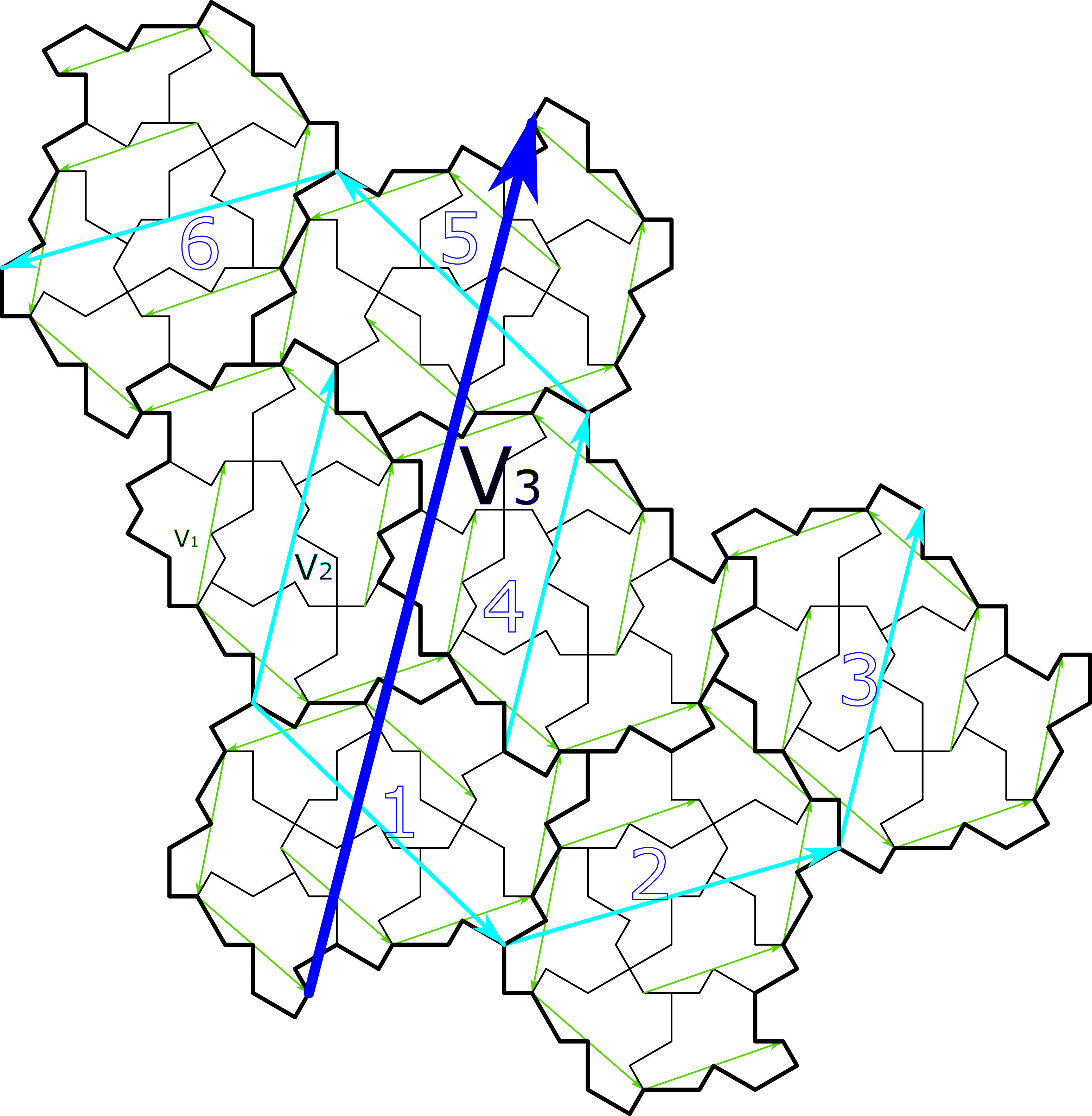}
        	\subcaption{Hat supertile \emph{hat-3}}
        	\label{fig:3a}
\end{minipage}
\begin{minipage}[b]{0.4\textwidth} 
\centering
	\includegraphics[scale = 0.25]{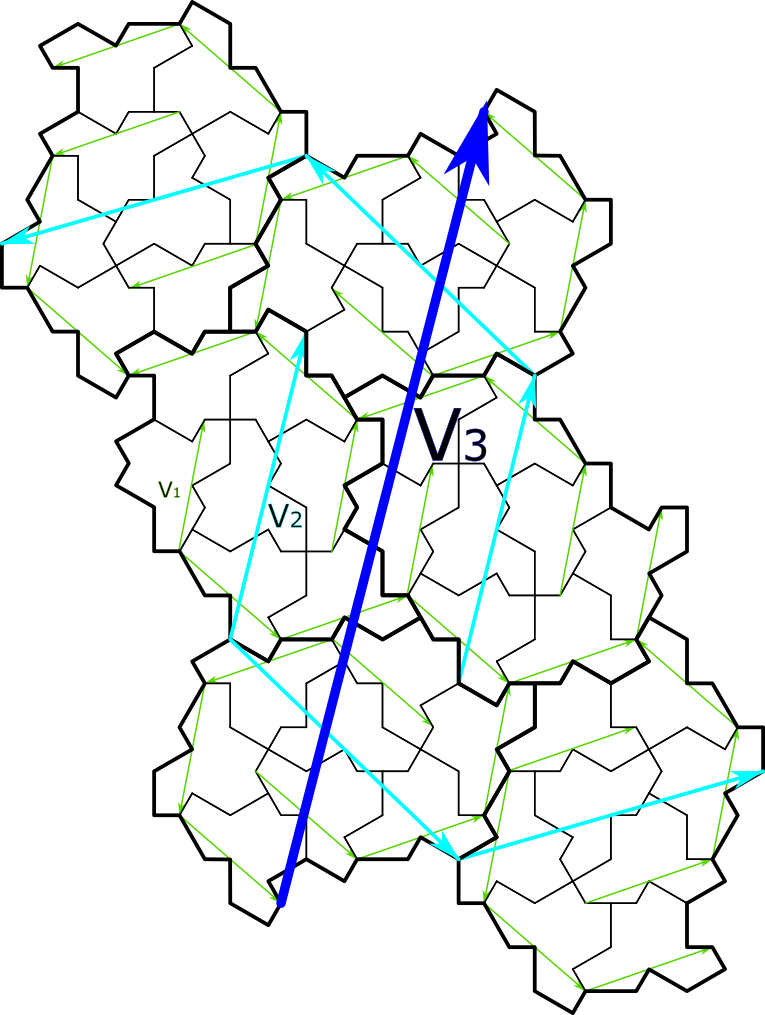}
        	\subcaption{Two-hat compound supertile \emph{thc-3}} 
	\label{fig:3b}
\end{minipage}
\caption{The third generation supertiles.}
\label{fig:3}
\end{figure}
\begin{itemize}
\item
In the $n$th supertile, thc-$(n-1)$ and the fourth hat-$(n-1)$ meet in the way that the thc-$(n-2)$ of the latter has its supervector where that of the missing third hat-$(n-2)$ of thc-$(n-1)$ would be. 
\item 
$V_n$ starts from the head of the third hat-$(n-2)$ of the first hat-$(n-1)$ and ends at the head of the third hat-$(n-2)$ of the fifth hat-$(n-1)$.
\end{itemize}

\section*{Supertile Rotaions}
At first glance, $V_1$, $V_2$ and $V_3$ look suspiciously parallel. A careful calculation shows they are not. In Figure~\ref{fig:4}, the hexagonal grid is drawn for hat-1 and hat-2. 
Let the short edge of the hat tile be 1, then it's not hard to see that
\begin{equation}
V_1 = (1, 3\sqrt{3}),\quad V_2 = (3, 7\sqrt{3}), \quad V_3 = (8, 18\sqrt{3}).
\end{equation}
\begin{figure}[h!tbp]
\centering
\begin{minipage}[b]{0.35\textwidth} 
\centering
	\includegraphics[scale = 1]{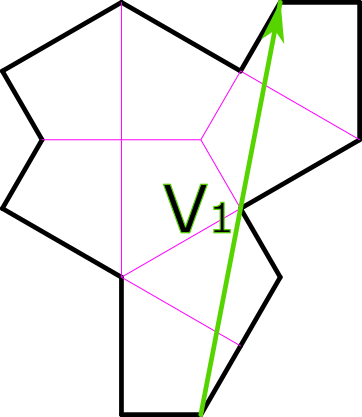}
        	\subcaption{Hat-1}
        	\label{fig:4a}
\end{minipage}
\hspace{3em}
\begin{minipage}[b]{0.35\textwidth} 
\centering
	\includegraphics[scale = 0.5]{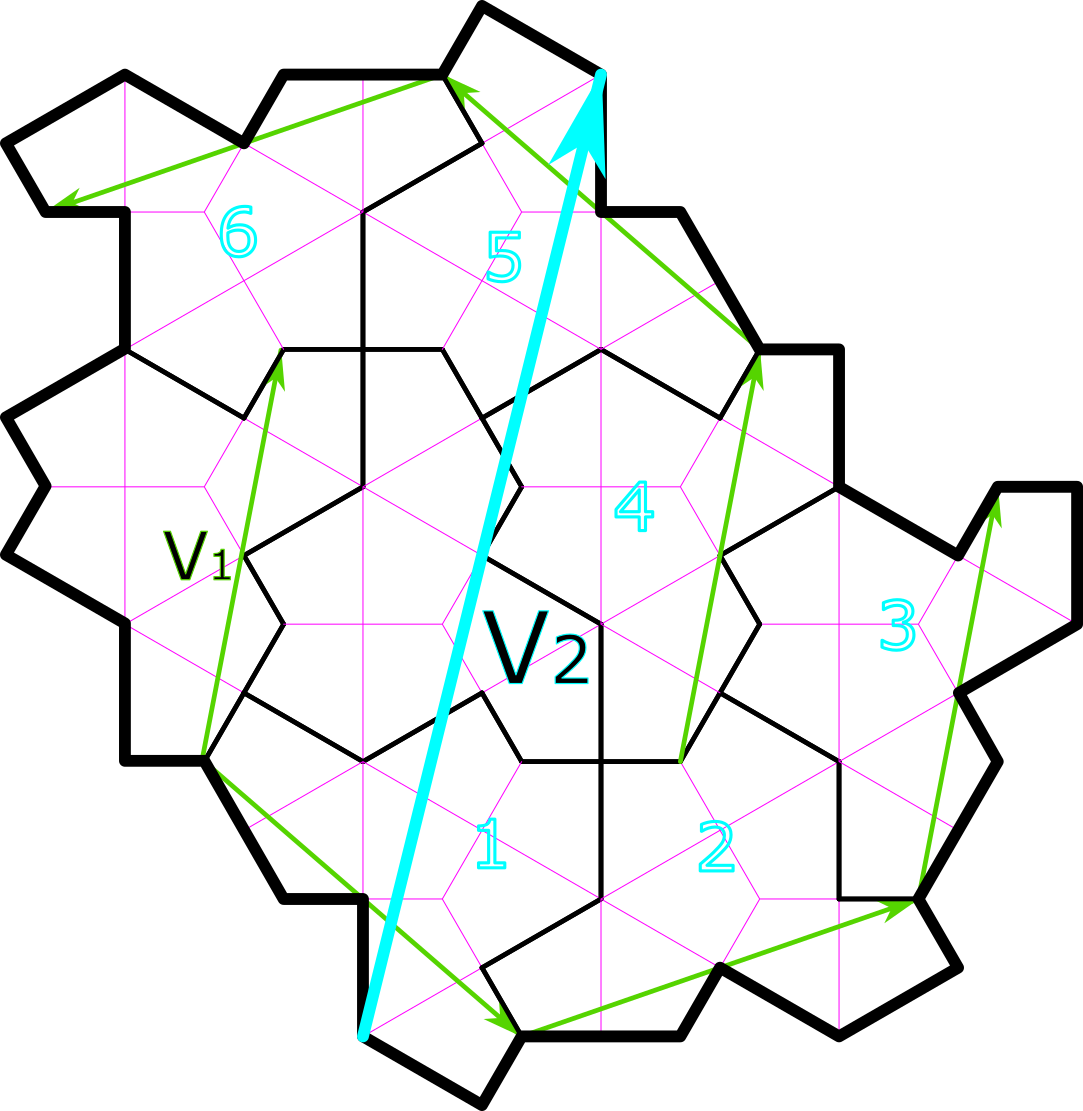}
        	\subcaption{Hat-2}
	\label{fig:4b}
\end{minipage}
\begin{minipage}[b]{0.8\textwidth} 
\centering
	\includegraphics[scale = 0.45]{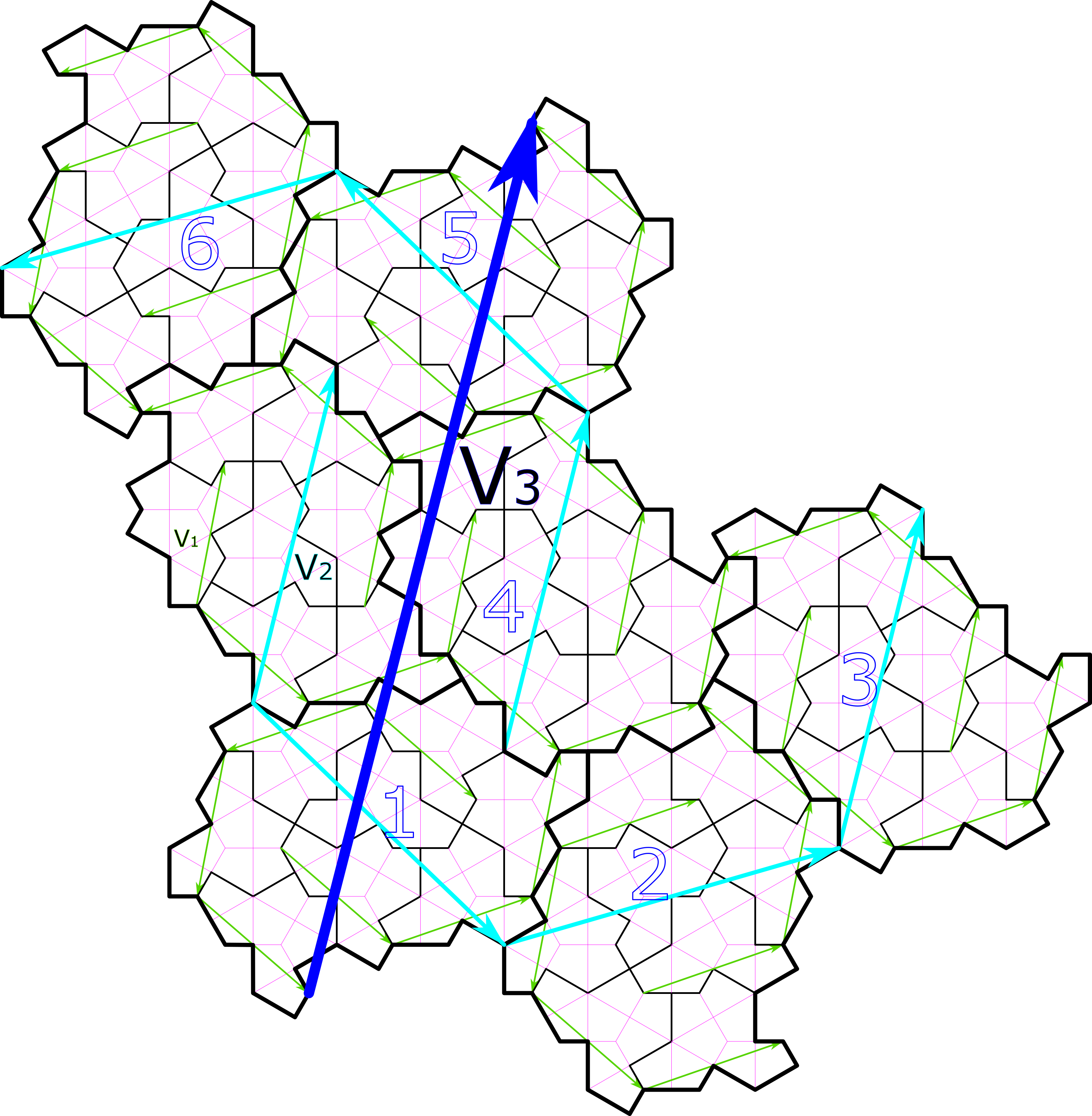}
        	\subcaption{Hat-3}
	\label{fig:4c}
\end{minipage}
\caption{Hat-1, hat-2, and hat-3 with hexagonal grid.}
\label{fig:4}
\end{figure}

Since the rule of tiling is recursive, involving \emph{two} previous generations, we expect there is a recurrence relation about $V_n$ that involves two previous terms. Indeed, we have a very simple relation
\begin{lemma}
\[V_n = 3 V_{n-1} - V_{n-2}, \quad \forall n\geq 3.\]
\label{lem:1}
\end{lemma}
\begin{proof}
We will prove by induction. For $n=3$, it is true. Assume it is so for some $n-1\geq 3$.  In the layout for hat-$n$, if we only focus on the supervectors of generations $n-3$ to $n$ (Figure~\ref{fig:5}), then we can see 
\[V_n = 2 X_{n-2} + W_{n-2} + U_{n-2} +V_{n-2} = 5 V_{n-2} + W_{n-2} + U_{n-2},\]
where we've used the fact that $X_{n-2} = 2 V_{n-2}$ from the half-hexagon shape. And because the fourth hat-$(n-1)$ rotates $120$ degrees counterclockwise with respect to the first hat-$(n-1)$, $Z_{n-2}$ is $W_{n-2}$ rotating $120$ degrees counterclockwise. Similarly, because the fifth hat-$(n-1)$ rotates $120$ degrees counterclockwise with respect to the fourth hat-$(n-1)$, $Y_{n-2}$ is $Z_{n-2}$ rotating $120$ degrees counterclockwise. Therefore, $U_{n-2}$ is $W_{n-2}$ rotating $60$ degrees counterclockwise.
\begin{figure}[h!tbp]
\centering
\begin{minipage}[b]{0.45\textwidth} 
\centering
	\includegraphics[scale = 0.1]{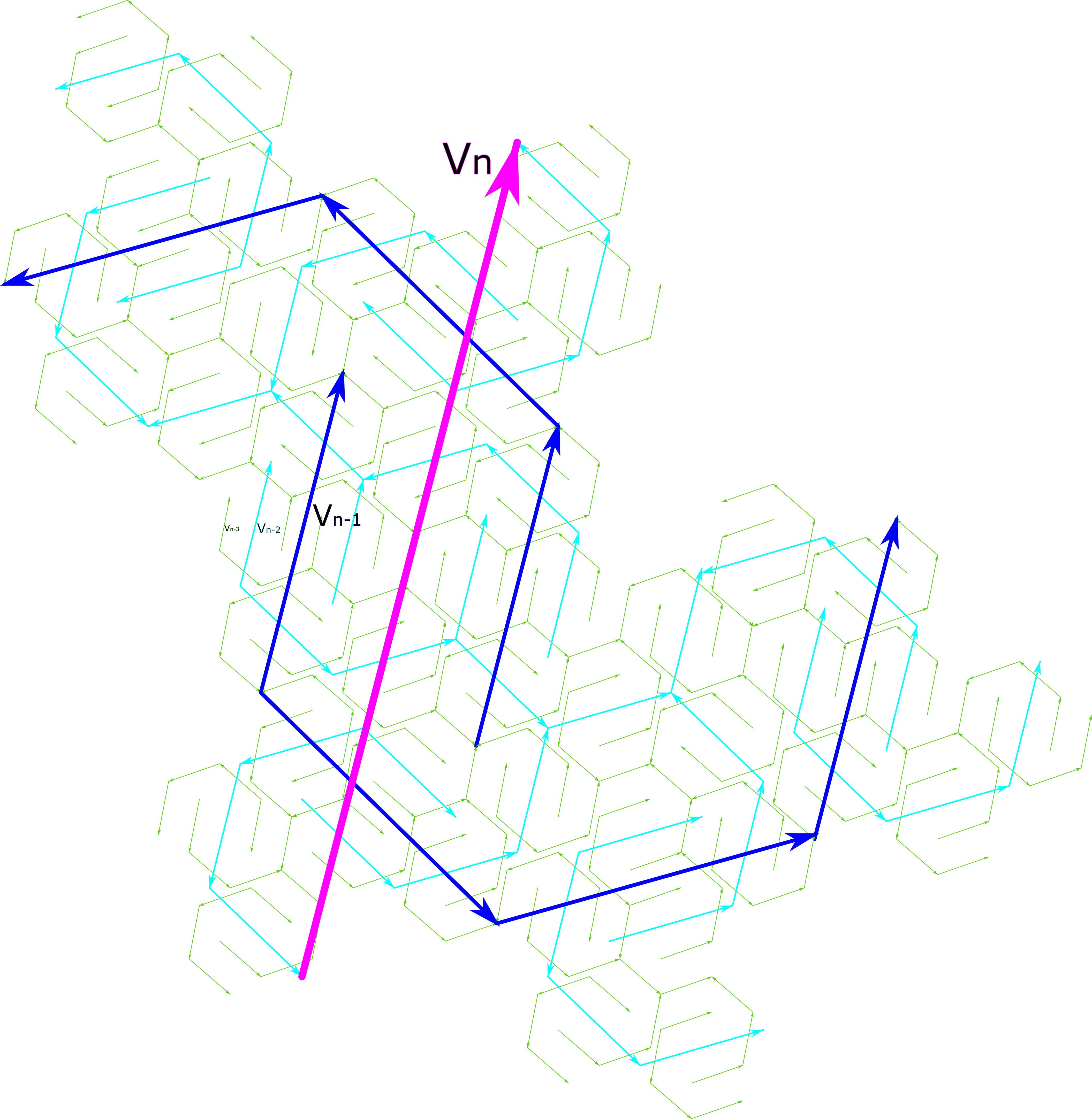}
        	\subcaption{Supervectors in the hat-$n$ layout}
        	\label{fig:5a}
\end{minipage}
\begin{minipage}[b]{0.45\textwidth} 
\centering
	\includegraphics[scale = 0.15]{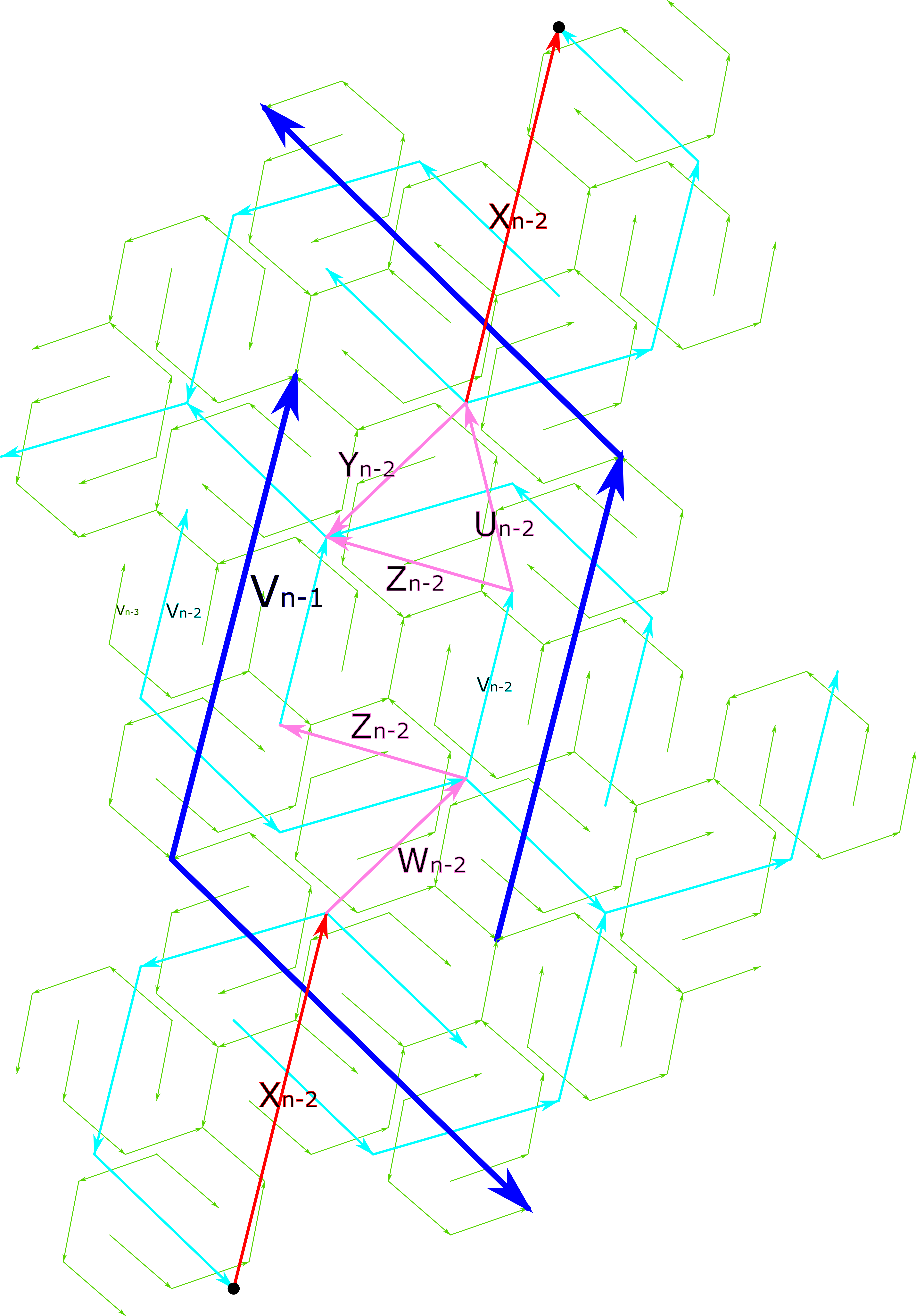}
        	\subcaption{Decomposition of $V_n$} 
	\label{fig:5b}
\end{minipage}
\caption{Supervectors in the hat-$n$ layout.}
\label{fig:5}
\end{figure}

For simplicity, we will use complex numbers for vector manipulation. Each $V_n$ has its corresponding complex number $v_n$, and likewise for the other vectors. Rotating $m\times 60$ degrees counterclockwise is simply multiplication by $e^{im\pi/3}$. We then have
\[v_n = 5 v_{n-2} + (1+e^{i\pi/3}) w_{n-2}.\]
To calculate $w_{n-2}$, let's move $W_{n-2}$ to a new location (Figure~\ref{fig:6a}). From this new location, we get
\[w_{n-2} + a_{n-2} = (1+ e^{-i\pi/3})v_{n-2}.\]
$A_{n-2}$ can be moved in parallel to a new location (Figure~\ref{fig:6b}), from which we get
\[a_{n-2} = (1+ e^{-i\pi/3})v_{n-3}.\]
\begin{figure}[h!tbp]
\centering
\begin{minipage}[b]{0.35\textwidth} 
\centering
	\includegraphics[scale = 0.25]{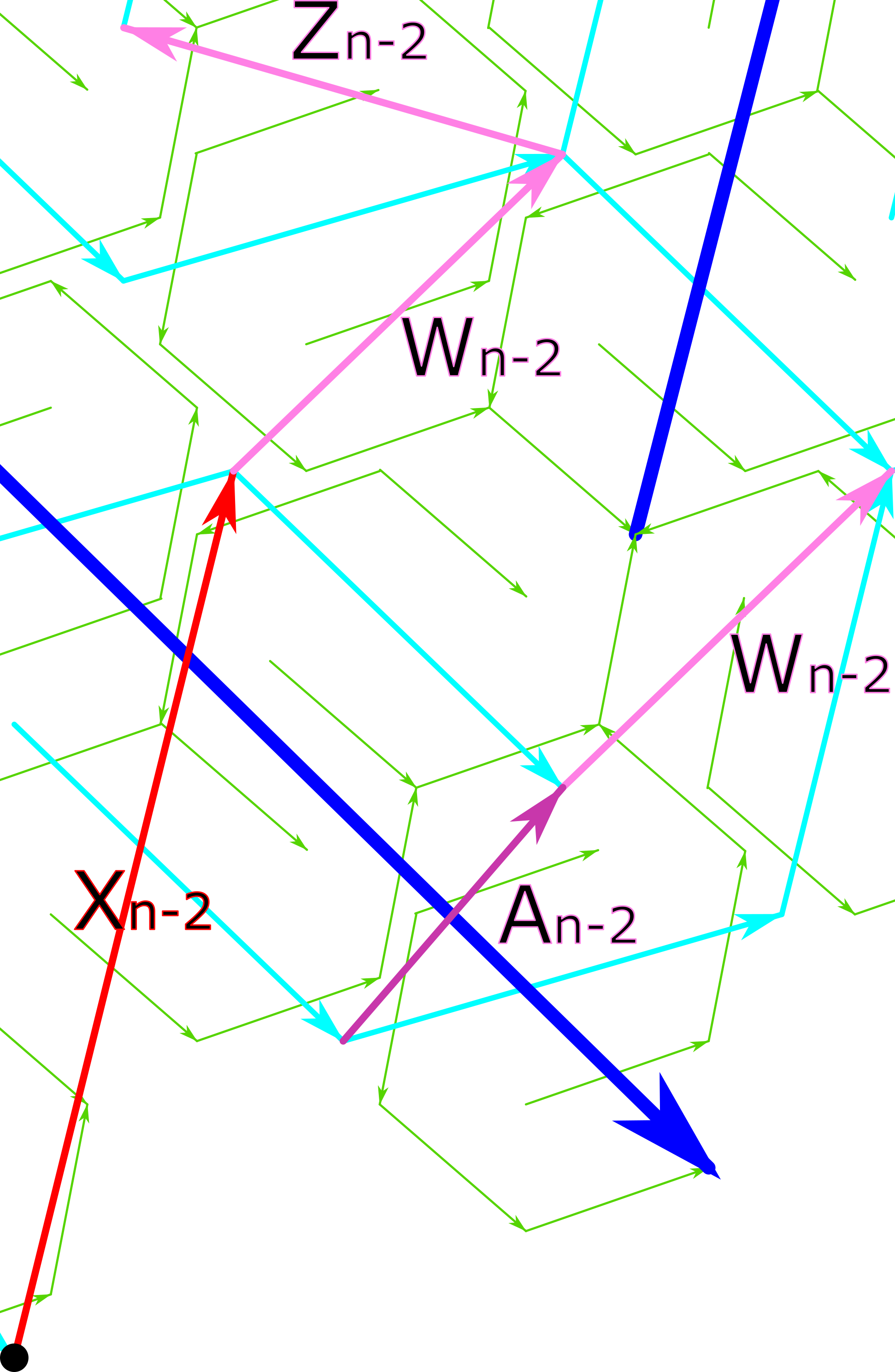}
        	\subcaption{$W_{n-2}$}
        	\label{fig:6a}
\end{minipage}
\hspace{3em}
\begin{minipage}[b]{0.35\textwidth} 
\centering
	\includegraphics[scale = 0.25]{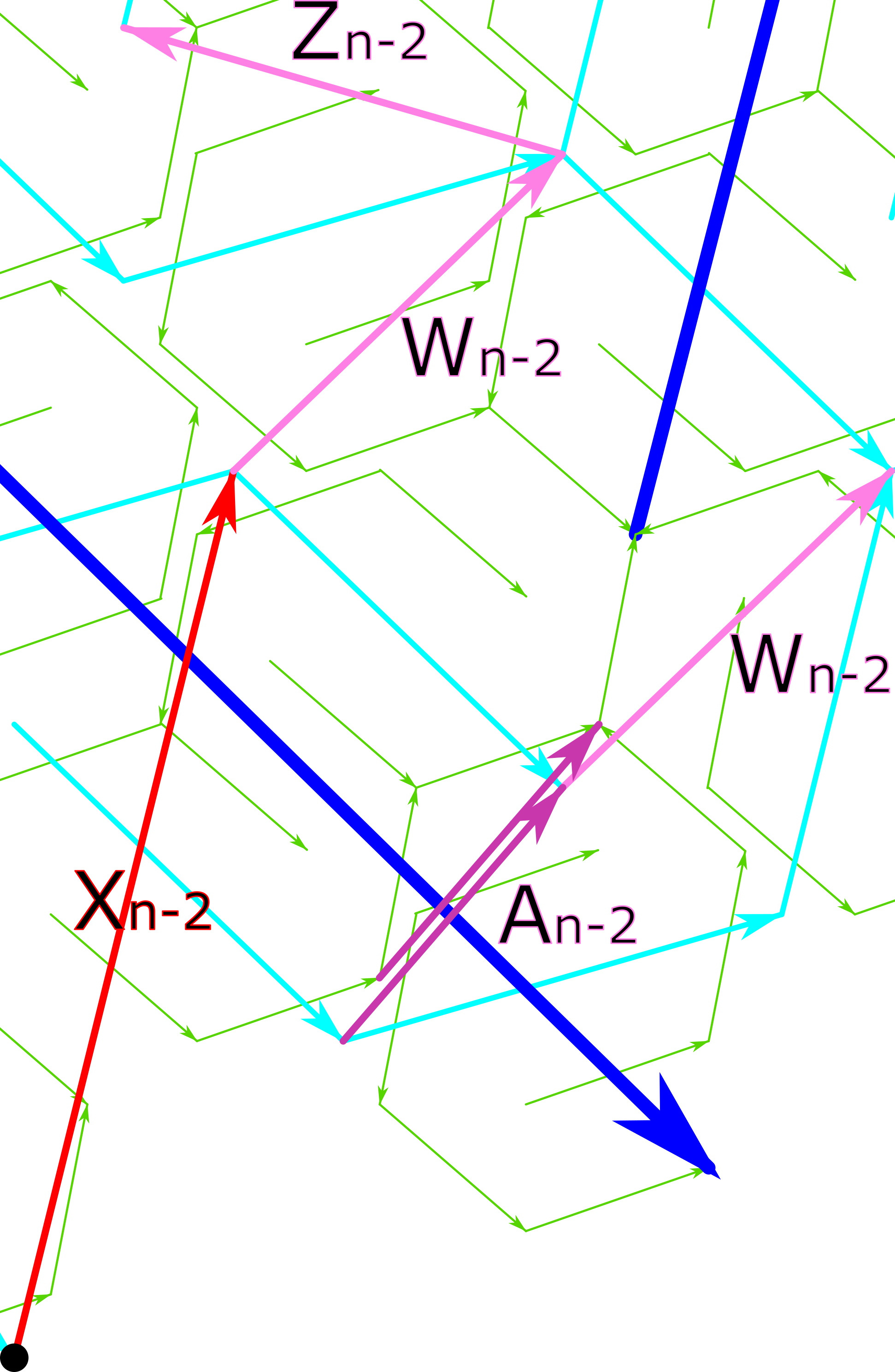}
        	\subcaption{$A_{n-2}$}
        	\label{fig:6b}
\end{minipage}
\caption{Parallel movement of vectors.}
\label{fig:6}
\end{figure}

Combining all facts, we get 
\[v_n = 5 v_{n-2} + (1+e^{i\pi/3})(1+ e^{-i\pi/3})(v_{n-2} - v_{n-3}) = 8 v_{n-2} -3 v_{n-3}.\]
By induction, 
\[v_{n-1} = 3 v_{n-2} - v_{n-3}.\]
Therefore,
\begin{equation}
v_{n} = 3 v_{n-1} - v_{n-2}.
\label{eqn:2}
\end{equation}
\end{proof}
We can start the sequence $\{V_n\}$ from $V_0 = 3V_1 - V_2 = (0, 2\sqrt{3})$.  Figures in this paper are oriented so that $V_0$ points up.

\section*{A Hidden Sequence}
From Lemma~\ref{lem:1} and the initial values $V_0$ and $V_1$, one can get all $V_n$. Once again we will work with complex numbers. 
The characteristic polynomial of Equation~(\ref{eqn:2}) 
has solutions 
\[\lambda_{1,2} = \frac12\left(3\pm\sqrt{5}\right) =\varphi^{\pm 2},\]
where $\varphi$ is the golden ratio. Hence
\[v_n = a \varphi^{2n} + b \varphi^{-2n}.\]
Since $v_0 = 2\sqrt{3} i$ and $v_1 = 1 + 3\sqrt{3}i$,
we have 
\[a = \frac{1}{\sqrt{5}} + \sqrt{3}i,\quad b = -\frac{1}{\sqrt{5}} + \sqrt{3}i,\]
and
\begin{equation}
v_n = \left(\frac{1}{\sqrt{5}} + \sqrt{3}i\right)\varphi^{2n} + \left(-\frac{1}{\sqrt{5}} + \sqrt{3}i\right)\varphi^{-2n} = F_{2n}+ \sqrt{3}L_{2n} i, 
\label{eqn:3}
\end{equation}
where $\{F_n\}$ is the Fibonacci sequence and $\{L_n\}$ is the Lucas sequence. 
The angle $V_n$ rotates from $V_0$ clockwise is 
\begin{equation}
\theta_n = \arctan\frac{\varphi^{2n}-\varphi^{-2n}}{\sqrt{15}(\varphi^{2n}+\varphi^{-2n})}
= \arctan \left(\frac{F_{2n}}{\sqrt{3}L_{2n}}\right), \quad \forall n\geq 0.
\label{eqn:theta}
\end{equation}
We have the total rotation of all generations 
\begin{equation}
\lim_{n\to\infty} \theta_n = \arccot(\sqrt{15}) = \arcsin\frac{1}{4}.
\label{eqn:5}
\end{equation}
The angle each $V_n$ rotates from the previous $V_{n-1}$ can be calculated as
\[
\alpha_n = \theta_n - \theta_{n-1} = \arccot \left(\sqrt{3}\frac{8L_{4n-2}+21}{15}\right), \quad \forall n\geq 1.
\]
Define
\begin{equation}
G_n  = \frac{8L_{4n-2}+21}{15}, \quad\forall n\geq 1.
\label{eqn:4}
\end{equation}
From the recurrence relation of $\{L_n\}$, it can be calculated that
\[G_n = 7 G_{n-1} - G_{n-2} - 7, \quad\forall n\geq 3, \]
with $G_1 = 3$ and $G_2 = 11$. $\{G_n\}$ is an integer sequence.
\begin{proposition} 
Let the short edge of the original hat tile be 1, the supervector 
\[ V_n = (F_{2n}, \sqrt{3}L_{2n}), \quad\forall n\geq 0.\]
The $n$-th generation of supertiles rotates from the $(n-1)$-th by an angle of 
\[\arccot(G_n\sqrt{3}),\]
with $G_n$ defined in Equation~(\ref{eqn:4}).
\end{proposition}
The first few terms of $G_n$ are 
\[3, 11, 67, 451, 3083, 21123, 144771, 992267, 6801091, 46615363, 319506443, 2189929731, 15010001667.\]

\section*{General Aperiodic Monotiles}
For a general aperiodic $\mathrm{Tile}(a,b)$, given the same alignment as the hat tile, following the red lines in Figure~\ref{fig:8}, we can calculate
\begin{figure}[h!tbp]
\centering
\begin{minipage}[b]{0.35\textwidth} 
\centering
	\includegraphics[scale = 1]{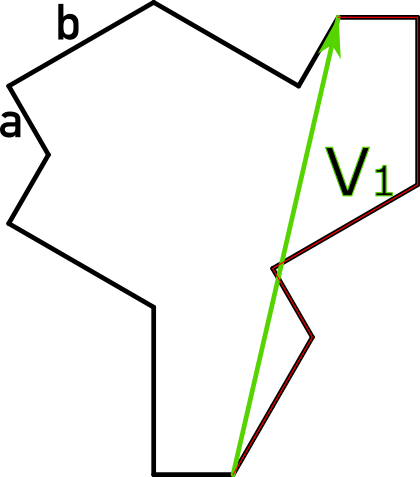}
        	\subcaption{$\mathrm{Tile}(a,b)$-1}
        	\label{fig:8a}
\end{minipage}
\hspace{3em}
\begin{minipage}[b]{0.35\textwidth} 
\centering
	\includegraphics[scale = 0.5]{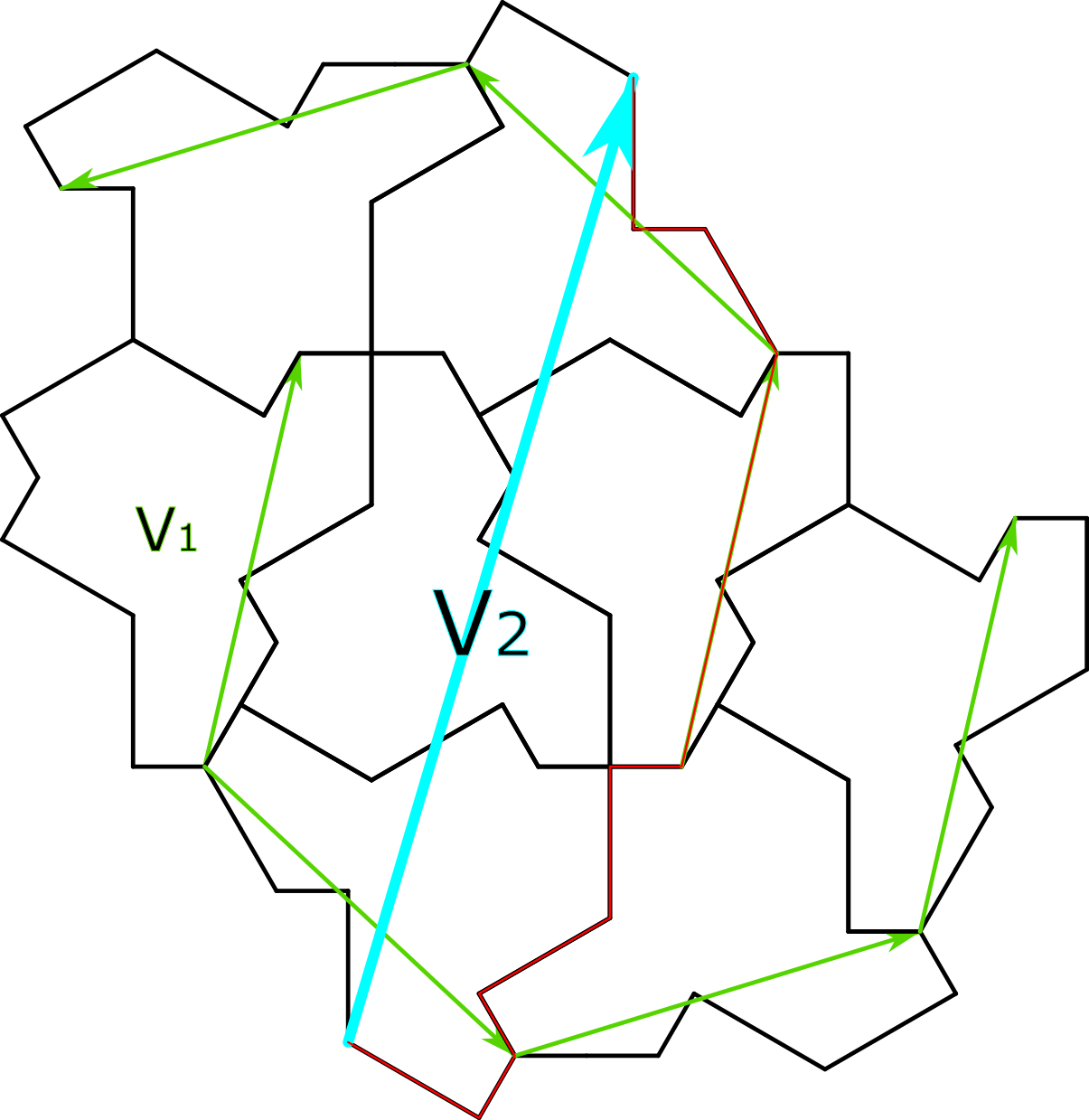}
        	\subcaption{$\mathrm{Tile}(a,b)$-2}
	\label{fig:8b}
\end{minipage}
\begin{minipage}[b]{0.8\textwidth} 
\centering
	\includegraphics[scale = 0.45]{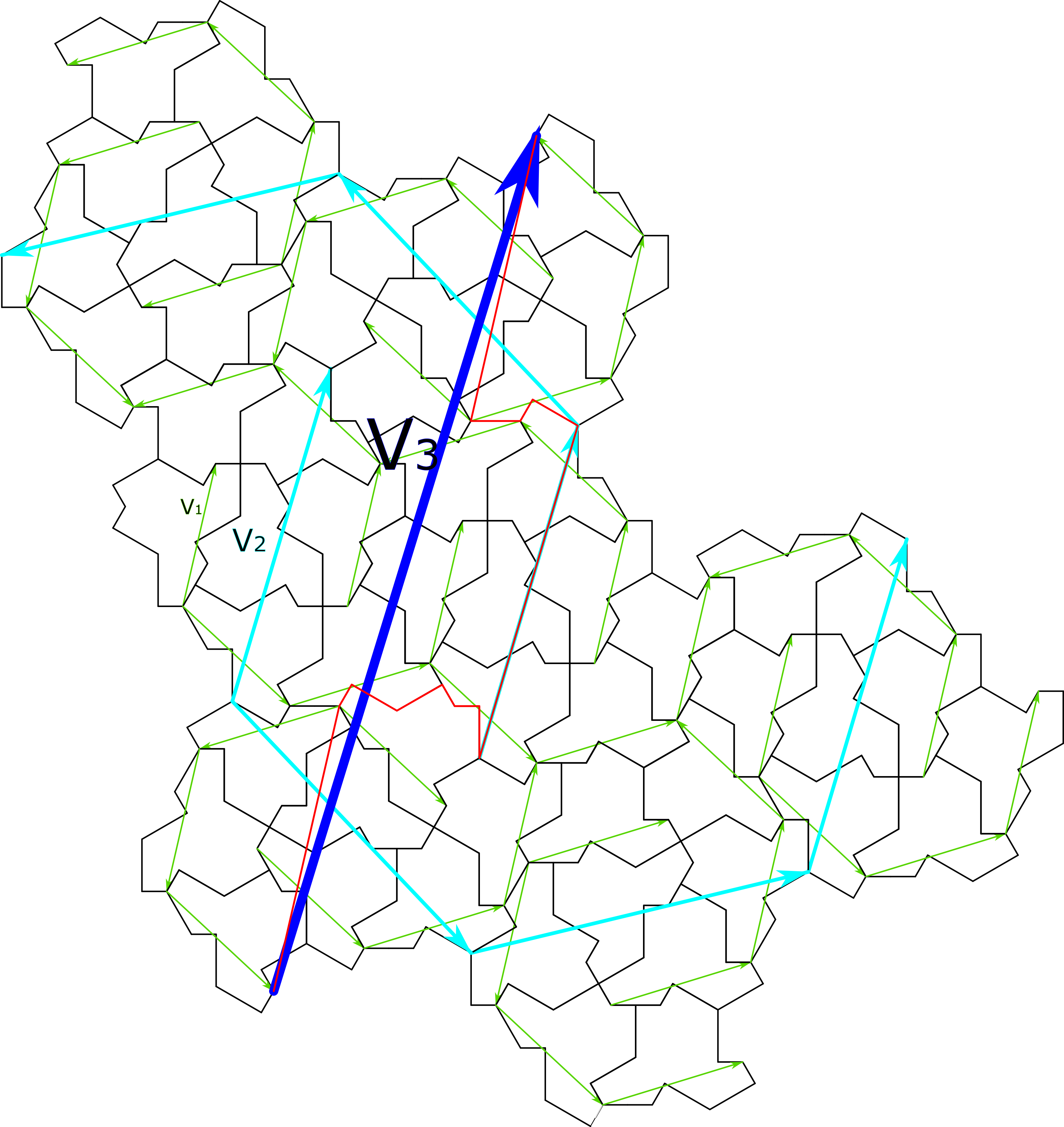}
        	\subcaption{$\mathrm{Tile}(a,b)$-3}
	\label{fig:8c}
\end{minipage}
\caption{The first three generations of supertiles for $\mathrm{Tile}(a,b)$.}
\label{fig:8}
\end{figure}
\[V_1 = \left(s, 3t\right),\quad V_2 = V_1 + \left(2s, 4t\right) = (3s, 7t),\]
and
\[V_3 = V_2 + 4 V_1 + \left(s, -t\right) = (8s, 18t),\]
where $s =(\sqrt{3}b-a)/2$ and $t = (\sqrt{3}a + b)/2$. It's straightforward to check 
\[V_3 = 3 V_2 - V_1.\]
Since the proof of Lemma~\ref{lem:1} only depends on this initial condition and the hexagonal structure of the supervectors, it is valid for $\mathrm{Tile}(a,b)$ as well.
\begin{lemma}
For any aperiodic $\mathrm{Tile}(a,b)$, with the supertiles and supervectors defined the same way as for the hat tile, we have 
\[V_n = 3 V_{n-1} - V_{n-2}, \quad \forall n\geq 3.\]
\label{lem:2}
\end{lemma}
\begin{figure}[h!tbp]
\centering
\begin{minipage}[b]{0.35\textwidth} 
\centering
	\includegraphics[scale = 1]{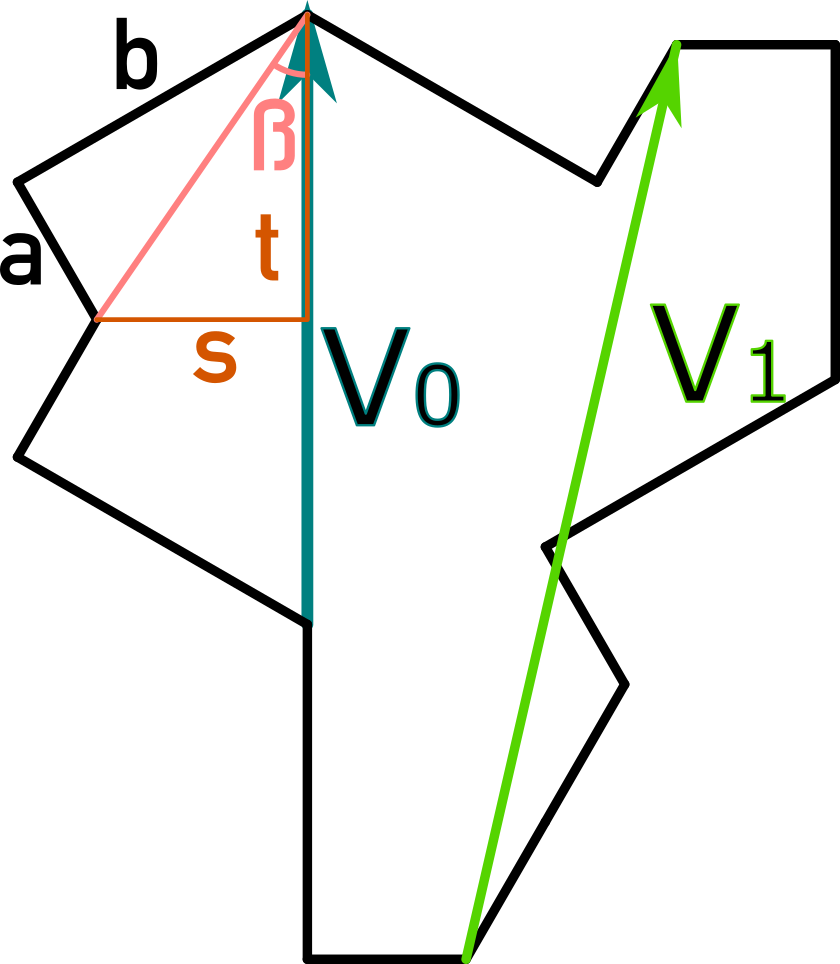}
        	\subcaption{$s, \beta>0$}
        	\label{fig:7a}
\end{minipage}
\hspace{3em}
\begin{minipage}[b]{0.35\textwidth} 
\centering
	\includegraphics[scale = 1]{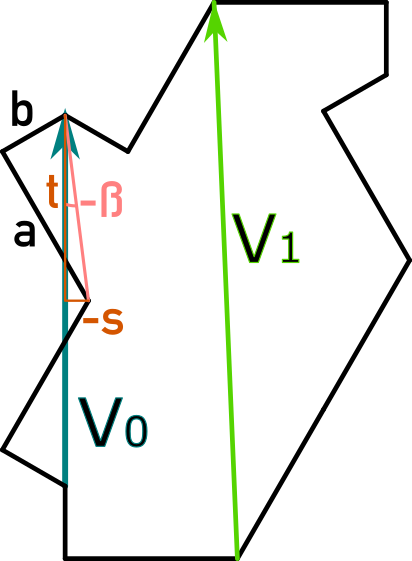}
        	\subcaption{$s, \beta<0$} 
	\label{fig:7b}
\end{minipage}
\caption{Marked $\mathrm{Tile}(a,b)$ tiles.}
\label{fig:7}
\end{figure}
Let $V_0 = 3V_1 - V_2 = (0,2t)$. Figure~\ref{fig:7} shows drawings of two generic cases where $s>0$ and $s<0$, respectively. Let $\beta = \arctan(s/t)$. $\beta$ is positive (negative resp.) when $V_1$ tilts to the right (left resp.) of $V_0$. We are considering the cases where $a>0, b>0$ and $a\neq b$. This implies $\beta \in (-\pi/6, \pi/12)\cup(\pi/12, \pi/3)$. When $\beta = 0$, $V_1$ align with $V_0$, hence all $\{V_n\}$ are parallel. 
 
The calculation of $\{V_n\}$ follows the same line as in the hat tile, and the result is
\begin{proposition} 
For any aperiodic $\mathrm{Tile}(a,b)$, with the supertiles and supervectors defined the same way as for the hat tile, we have 
\[ V_n = \left(F_{2n}~s, L_{2n}~t\right), \quad\forall n\geq 0.\]
The $n$-th generation of supertiles rotates clockwise from the $(n-1)$-th by an angle of 
\[\arctan\left(s/(G_n~t)\right) = \arctan\left(\tan\beta/G_n\right).\]
The total rotation of all generations is
\[\arctan\left(s/(\sqrt{5}~t)\right) = \arctan\left(\tan\beta/\sqrt{5}\right).\]
\end{proposition}

If we add $V_0$ to the two-tile-compound as in Figure~\ref{fig:9a}, then Figure~\ref{fig:9b} shows that $V_0$ form the same hexagonal structure as $V_1$ in the second generation supertile. However, the structure of $V_0$ is not exactly the same as all other $V_n$'s, as Figure~\ref{fig:9c} shows. 
\begin{figure}[h!tbp]
\centering
\begin{minipage}[b]{0.4\textwidth} 
\centering
	\includegraphics[scale = 1]{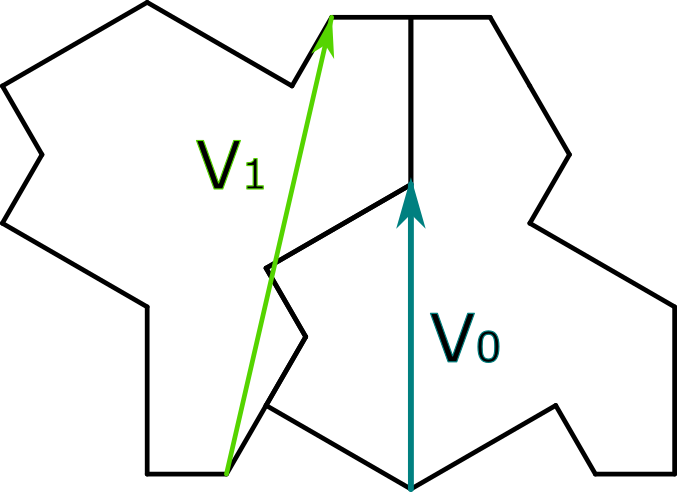}
        	\subcaption{Marked two-$\mathrm{Tile}(a,b)$-compound}
        	\label{fig:9a}
\end{minipage}
\begin{minipage}[b]{0.5\textwidth} 
\centering
	\includegraphics[scale = 0.6]{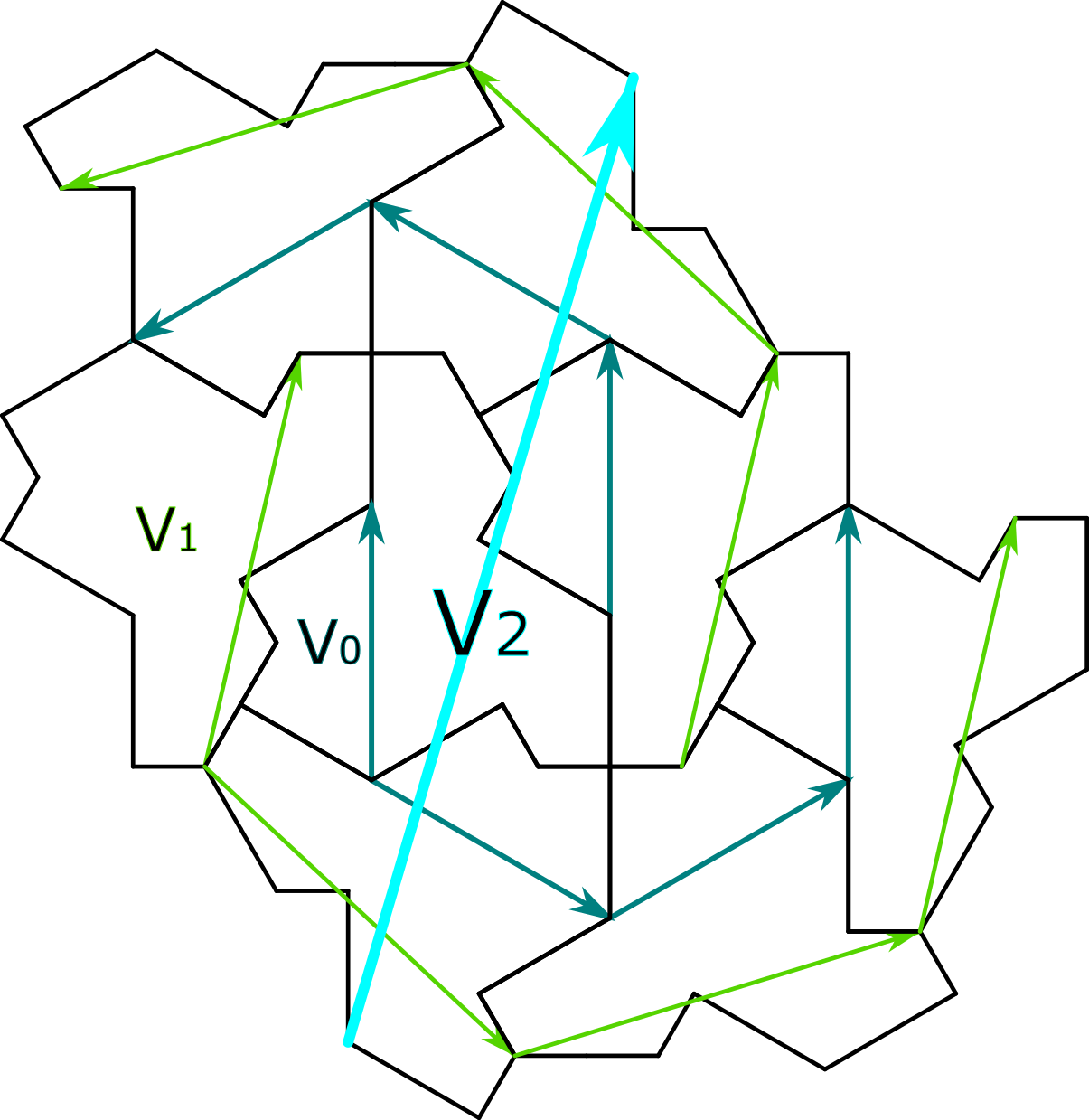}
        	\subcaption{Second generation supertile} 
	\label{fig:9b}
\end{minipage}
\begin{minipage}[b]{0.8\textwidth} 
\centering
	\includegraphics[scale = 0.45]{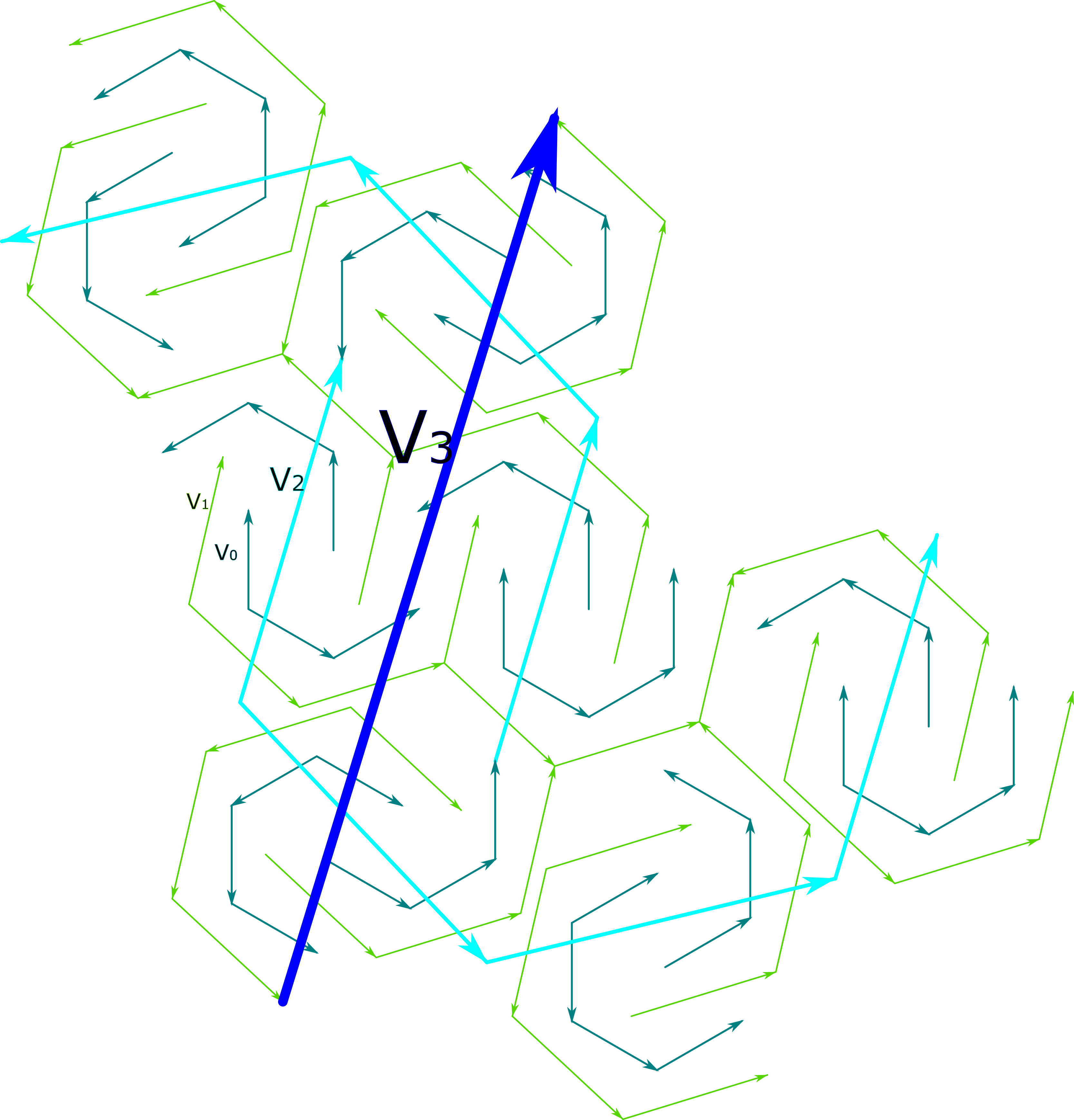}
        	\subcaption{Supervectors in the third generation supertile} 
	\label{fig:9c}
\end{minipage}
\caption{Hexagonal structure for $V_0$.}
\label{fig:9}
\end{figure}
\section*{Conclusions}
The supervectors of the supertiles take an elegant form involving Fibonacci and Lucas sequences, indicating yet another deep combinatorial meaning of aperiodic monotiles in the hat family. A new integer sequence linearly related to the Lucas sequence is in the rotational angles of the supertiles. In the fractal limit, the supertiles scale by a factor of $\varphi^2$ after each iteration, and the tangent of the incremental rotational angles scales by a factor of $\varphi^{-4}$. The meaning of the tangent of the such angles need to be explored.

\section*{Acknowledgments}
Thanks to G4G15 hands-on activities and Chaim Goodman-Strauss for inspiring this work and Haoran Chen for providing feedback.

    
{\setlength{\baselineskip}{13pt} 
\raggedright			

} 
\end{document}